\documentclass[11pt]{amsart}

\usepackage[utf8]{inputenc}
\usepackage{amsmath, amssymb,amsthm,tikz}
\usepackage{tikz-cd}
\usepackage{tikz}
\usepackage{todonotes}
\usepackage{mathrsfs}
\usepackage{extarrows}
\usepackage{graphicx}
\usepackage{thmtools}
\usepackage{mathtools}
\usepackage{hyperref}
\usepackage{multirow}
\usepackage{esvect} 

\usepackage{nicematrix}
\usepackage{adjustbox}
\usepackage[margin=1.48in]{geometry}

\usepackage[english]{babel}
\usepackage{comment}

\usepackage[toc,page]{appendix}

\newcommand{\ZZ}{\mathbb{Z}}

\newcommand{\CC}{\mathbb{C}}

\newcommand{\Qp}{{\mathbb{Q}_p}}

\newcommand{\FF}{\mathbb{F}}

\newcommand{\LLC}{{\mathrm{LLC}}}

\newcommand{\Fq}{{\mathbb{F}_q}}

\newcommand{\cA}{\mathcal{A}}
\newcommand{\cC}{\mathcal{C}}
\newcommand{\cD}{\mathcal{D}}
\newcommand{\cH}{\mathcal{H}}

\newcommand{\Fr}{\mathrm{Frob}}
\newcommand{\reg}{\mathrm{reg}}
\newcommand{\Sh}{\mathrm{Sh}}
\newcommand{\rI}{\mathrm{I}}
\newcommand{\sw}{\mathrm{sw}}

\newcommand{\cE}{\mathcal{E}}

\newcommand{\cM}{\mathcal{M}}

\newcommand{\cT}{\mathcal{T}}

\newcommand{\fP}{\mathfrak{P}}

\newcommand{\End}{\operatorname{End}}

\newcommand{\Ind}{\operatorname{Ind}}

\newcommand{\Mod}{\operatorname{Mod}}

\newcommand{\bG}{\mathbf{G}}
\newcommand{\bL}{\mathbf{L}}
\newcommand{\bM}{\mathbf{M}}

\newcommand{\bT}{\mathbf{T}}

\newcommand{\diag}{\operatorname{diag}}

\newcommand{\tame}{\operatorname{t}}

\newcommand{\unip}{\mathrm{u}}

\newcommand{\Tr}{\operatorname{Tr}}

\newcommand{\Hom}{\operatorname{Hom}}

\newcommand{\Cent}{\operatorname{C}}

\newcommand{\fR}{\mathfrak{R}}
\newcommand{\fB}{\mathfrak{B}}

\newcommand{\Irr}{\operatorname{Irr}}

\newcommand{\spe}{\operatorname{sp}}

\newcommand{\fo}{\mathfrak{o}}
\newcommand{\fp}{\mathfrak{p}}

\newcommand{\ft}{\mathfrak{t}}

\newcommand{\fX}{{\mathfrak{X}}}

\newcommand{\abel}{{\mathrm{ab}}}

\newcommand{\rN}{{\mathrm{N}}}

\newcommand{\kbar}{{\overline{k}}}

\newcommand{\Gal}{\operatorname{Gal}}

\newcommand{\GL}{\mathrm{GL}}

\newcommand{\cusp}{\mathrm{c}}

\def\tr{\mathrm{tr}}

\hypersetup{
    colorlinks,
    citecolor=blue,
    filecolor=blue,
    linkcolor=blue,
    urlcolor=blue
}

\title[]{On the Macdonald correspondence}

\author{Anne-Marie Aubert}
\address{Sorbonne Universit\'e and Universit\'e Paris Cit\'e, CNRS,
IMJ-PRG, F-75005 Paris, France}
\email{anne-marie.aubert@imj-prg.fr}

\date{\today}

\numberwithin{equation}{subsection}

\newtheorem{prop}[equation]{Proposition}
\newtheorem{thm}{Theorem}

\newtheorem{cor}[equation]{Corollary}
\newtheorem*{thm*}{Theorem}
\newtheorem*{thmconverse}{Converse Theorem}

\theoremstyle{definition}
\newtheorem{defn}[equation]{Definition}
\newtheorem{remark}[equation]{Remark}

\def\spe{{\mathrm{sp}}}

\def\WD{{\mathrm{WD}}}

\begin{document}

\maketitle

 \centerline{\sl To the memory of Ian G. Macdonald}

\begin{abstract}  In 1980  Ian G.  Macdonald established an explicit bijection between the isomorphism classes of the irreducible representations of $\GL_n(k)$,  where $k$ is a finite field, and inertia equivalence classes of
admissible tamely ramified $n$-dimensional Weil-Deligne representations of $W_F$, where $F$ is a non-archimedean local field with residue field $k$ and $W_F$ the absolute Weil group of $F$.   

We describe a construction of the Macdonald correspondence based on the specialization to $\GL_n(k)$ of Lusztig's classification of irreducible representations of finite groups of Lie type, and review some 
properties of the correspondence. We define $\epsilon$-factors for pairs of irreducible cuspidal representations of finite general linear groups, and show that they match with the expected Deligne $\epsilon$-factors under the 
Macdonald correspondence. We use these $\epsilon$-factors for pairs  to obtain a characterization of  the Macdonald correspondence for the  irreducible cuspidal representations.
\end{abstract}

\tableofcontents

\section{Introduction}
Let $n$ be a positive integer and let $\GL_n(k)$  be the general linear group of $n$ by $n$ matrices over a finite field $k$. In \cite{Gr} Green determined all irreducible characters of  $\GL_n(k)$ 
 by parabolic induction, generalizing the Frobenius theory of the symmetric groups. The Green theory has greatly influenced the developments of representation theory and algebraic combinatorics. 
 Deligne and Lusztig \cite{DL} generalized the Green parabolic induction to construct linear representations of finite groups of Lie types using $\ell$-adic cohomology  and later 
 Lusztig found all representations of finite  groups of Lie type \cite{Lus}.
 
In algebraic combinatorics, the Green theory showed the importance of  a new family of symmetric functions-Hall-Littlewood functions. 
Macdonald has written the famous book \cite{Mac2} to account the Green identification of the Hall algebra with the ring of symmetric functions. 

The (non-archimedean) local Langlands correspondence was established by Laumon, Rapoport, and  Stuhler \cite{LRS} for function fields, and by  Harris-Taylor \cite{HT}, Henniart \cite{He-2000} and
 Scholze \cite{Sc},  independently, for finite extensions of $\Qp$, with $p$ a prime number. 
Long before this,  in 1980, by using the Green construction, Macdonald 
constructed in \cite{Mac} a correspondence between the isomorphism classes of the irreducible representations of $\GL_n(k)$, and inertia equivalence classes of
admissible tamely ramified $n$-dimensional Weil-Deligne representations of $W_F$, where $F$ is a non-archimedean local field with residue field $k$ and $W_F$ the absolute Weil group of $F$.  
Macdonald  defined also in \cite{Mac} a version over finite fields of the Godement-Jacquet $\epsilon$-factors and proved that his  correspondence matches them with the $\epsilon$-factors  that were 
defined by Deligne in \cite{Del}. 

Silberger and Zink established  in \cite[(A.1)]{SiZi2} the compatibility of the Macdonald correspondence with the local Langlands correspondence for the depth-zero representations of $\GL_n(F)$.

On the other hand, in the $p$-adic situation, the pioneering work \cite{JPSS} of Jacquet, Piatetski-Shapiro and Shalika defined the notion of Rankin-Selberg local factors, also called \textit{local factors for pairs}, for 
pairs irreducible generic representations of $\GL_n(F)$, including $\gamma$-factors for pairs. In 1993, Henniart established in \cite{He-1993} a converse theorem based on these latter, and, as consequence, obtained a 
characterization of the local Langlands correspondence.

In the finite situation, $\gamma$-factors for pairs were defined by Roditty-Gershon in \cite{Rod} and  by Nien in \cite[Theorem~2.10]{Nie-2014}, but only when at least one the representation in the pair is cuspidal. 
Recently in \cite{SoZe}, Soudry and Zelingher provided  a general definition of $\gamma$-factors for pairs of irreducible generic representations of finite general linear groups, by using a finite version  of the Langlands-Shahidi method, 
and established that these  $\gamma$-factors for pairs  are multiplicative.

We first describe a construction of the Macdonald correspondence based on the specialization to $\GL_n(k)$ of Lusztig's classification of irreducible representations of finite groups of Lie type, and review 
properties of the correspondence. Next, by using  \cite{SoZe}, we define $\epsilon$-factors for pairs of irreducible cuspidal representations  of finite general linear groups, and show that they match with the expected Deligne $\epsilon$-factors under the 
Macdonald correspondence. Thanks to a converse theorem established by Nien in \cite{Nie-2014}, following Henniart's approach, we obtain  in Theorem~\ref{thm:cuspLC} a characterization of the Macdonald correspondence for the irreducible cuspidal representations.

\section{Construction of the Macdonald correspondence} \label{se:const}

\subsection{The set \texorpdfstring{$\fP(\fX)^\Gamma$}{\fP(\fX)^\Gamma}}  \label{sec:fP}
Let $p$ be a fixed prime number and let $k=\Fq$ be a finite  field with $q$ elements, where $q$ is a power of $p$. 
Let $\kbar$ be an algebraic closure of $k$, and let $\Fr\colon x\mapsto x^q$  denote the Frobenius automorphism of $\kbar$ over $k$.
For each integer $n\ge 1$, the set $k_n$ of fixed points of $\Fr^n$ in $\kbar$ is the unique extension of $k$ of degree $n$ contained in $\kbar$. 
We denote by $\Gamma_n$ the Galois group $\Gal(k_n/k)$ and by $\Gamma$ the Galois group $\Gal(\kbar/k)$.  We have $\Gamma\simeq\widehat\ZZ:=\displaystyle\lim_{\longleftarrow} \ZZ/n\ZZ$.

Let $\kbar^\times$ be the multiplicative group of $\kbar$, and $k_n^\times$ the  multiplicative group of $k_n$, so that $k_n^\times=(\kbar^\times)^{\Fr^n}$.
If $m$ divides $n$, the norm map $\rN_{n,m}\colon k_n^\times\to k_m^\times$, defined by
\begin{equation}
\rN_{n,m}(x):=x^{(q^n-1)/(q^m-1)}=\prod_{i=0}^{d-1}\Fr^{mi}x,
\end{equation}
where $d:=n/m$, is a surjective homomorphism. The groups $k_n^\times$ and the homomorphisms $\rN_{n,m}$ form an inverse system. 

Their inverse limit $\displaystyle\lim_{\longleftarrow} k_n^\times$ is a profinite group. The character group of $\fX$ of the latter is therefore a discrete group. Whenever $m$ divides $n$, the character group $\fX_m$ of $k^\times_m$  is embedded in the character group $\fX_n$ of $k^\times_n$  by the transpose of the norm homomorphism $\rN_{n,m}$,
 and we have
\begin{equation} \label{eqn:fX}
\fX=\lim_{\longrightarrow} \fX_n.
\end{equation}
Every element in $\fX$ has finite order, and the group $\fX$ is (non-canonically) isomorphic to $\kbar^\times$. 

A character of $k_n^\times$ is said to be \textit{regular} if it is fixed by no nontrivial element of $\Gamma_n$.  We denote by $\fX_n^\reg$ the set of regular characters of $k_n^\times$. 
Thus, we have $\theta\in\fX_n$ if and only if the cardinality of the set $\left\{\theta,\theta^q,\ldots,\theta^{q^{n-1}}\right\}$ is $n$.

Let  $\fP(n)$ denote the set of partitions of $n$, and $\fP:=\bigcup_{n\in\ZZ_{\ge 0}}\fP(n)$. For a partition $\lambda\in\fP$, we have $\lambda\in\fP(n)$ for some integer $n$; we define $|\lambda|:=n$.  We consider the set
\begin{equation} \label{eqn:fP}
\fP(\fX):=\left\{\underline{\lambda}\colon\fX\to\fP\,:\, \theta\mapsto \lambda_\theta:=\underline{\lambda}(\theta)\right\}
\end{equation}
of partition-valued functions $\underline{\lambda}$ on $\fX$ such that $\underline{\lambda}(\theta)=(0)$ except for finitely many $\theta\in\fX$. 
Let $\fP_n(\fX)$ denote the set  of partition-valued functions $\underline{\lambda}\colon\fX\to\fP$ such that $\sum_{\theta\in \fX}|\lambda_\theta|=n$.

The group $\Gamma$ acts naturally on $P(\fX)$. 
Let $\fP(\fX)^\Gamma$ be the subset consisting of all $\Gamma$-invariant functions on $\fP(\fX)$. Since $\underline{\lambda}\in\fP(\fX)^\Gamma$ is constant on any Galois orbit $[\theta]:=\Gamma\cdot\theta$ (for $\theta\in\fX$), we may view the elements of 
$\fP(\fX)^\Gamma$ as partition-valued functions defined on the set of Galois orbits $\Gamma\backslash\fX$. We write $\fP_n(\fX)^\Gamma:=\fP(\fX)^\Gamma\cap \fP_n(\fX)$.

\subsection{Weil-Deligne representations}  \label{sec:WD}
Let $\WD_F:=W_F\ltimes\CC$ be the Weil-Deligne group of $F$. The semidirect product is defined via the action
\begin{equation}
wxw^{-1}=|\!|w|\!| x,\quad\text{for $w\in W_F$ and $x\in \CC$,}
\end{equation}
where $|\!|w|\!|$ is the norm of $w\in W_F$. Let $I_F$ and $P_F$ denote the inertia and the wild inertia  subgroups of $W_F$ respectively. These subgroups are normal in $\WD_F$.
We denote by $W_F^{\tame}:=W_F/P_F$ and $I_F^{\tame}:=I_F/P_F$ the \text{tame Weil group} and the \text{tame inertia subgroup} of $F$, respectively.  We have
\begin{equation} \label{eqn:I/P}
I_F^{\tame}\cong \lim_{\longleftarrow} k_n^\times. 
\end{equation}
It follows that the character group of  $I_F^{\tame}$ is canonically isomorphic to $\fX$. 

A \textit{Weil-Deligne} representation of $W_F$ is  a pair $(\rho,N)$, where $\rho$ is a finite-dimensional representation of $W_F$ acting  on a $\CC$-vector space $V_\rho$, such that $\ker\rho$ contains an open subgroup of $I_F$ and $\rho(w)$ is semisimple for any $w\in W_F$, and $N$ is  a nilpotent endomorphism of $V_\rho$ satisfying
\begin{equation} \label{eqn:relation}
\rho(w)N\rho(w)^{-1}=|\!|w|\!|N\quad\text{for any $w\in W_F$.}
\end{equation}
The representation $(\rho,N)$ is said to be:
\begin{itemize}
\item \textit{unramified} if $I_F \subset \ker \rho$, 
\item \textit{tamely ramified} if $P_F \subset \ker(\rho)$. 
\end{itemize}

The irreducible Weil-Deligne representations $(\rho,N)$ are those such that $\rho$ is an irreducible representation of $W_F$ and $N=0$. We denote by $\Irr(W_F^{\tame}$ the set of irreducible tamely ramified representations of $W_F$.

For every integer $m\ge 1$,  the \textit{special representation} of dimension $m$ is the representation $\spe_F(m):=(\rho_m,N_m)$, where 
\begin{equation} \label{eqn:special}
\rho_m(w):=\diag(|\!|w|\!|^{(1-m)/2},|\!|w|\!|^{(3-m)/2},\ldots,|\!|w|\!|^{(m-3)/2},|\!|w|\!|^{(m-1)/2}),
\end{equation}
and $N_m$ is the nilpotent $m\times m$-matrix with $1$s on the subdiagonal and $0$s elsewhere. The pair $(\rho_m,N_m)$ clearly satisfy \eqref{eqn:relation}. 

The indecomposable  Weil-Deligne representations are those of the form
\begin{equation} \label{eqn:indec}
\spe_F(m)\otimes\rho=(\diag(N_m,\ldots,N_m),\rho_m\otimes \rho),
\end{equation}
where $m\ge 1$ and  $\rho$ is an irreducible representation of $W_F$.

If $\lambda=(n_1\ge n_2\ge\cdots  \ge n_r>0)$ is a nonzero partition of $n$, we set
\begin{equation}
\spe_F(\lambda):=\spe_F(n_1)\oplus\cdots\oplus\spe_F(n_r).
\end{equation}
\begin{defn} \label{defn:I-equivalence} {\rm \cite[\S3]{Mac}}
Two Weil-Deligne representations  are \textit{inertially equivalent} if their restriction to $I_F\times\CC\subset \WD_F$ are equivalent.
\end{defn}
Let  $\Irr(W_F^{\tame})_{I_F}$ denote the set of inertial equivalence classes of rreducible tamely ramified representations of $W_F$.
If $\theta\in\fX$, let $\rho(\theta)$ be an irreducible representation of $W_F^{\tame}$ such that $\theta$ occurs in the restriction of $\rho(\theta)$ to $I_F$. We denote by $\rho(\theta)_{I_F}$ the inertial class of $\rho(\theta)$. The we have a bijective map
\begin{eqnarray} \label{eqn:bij}
\Gamma\backslash \fX&\longrightarrow&\Irr(W_F^{\tame})_{I_F}\cr
\theta&\mapsto&\rho(\theta)_{I_F}.
\end{eqnarray}
Let $\fR(\WD_F^{\tame})_{I_F}$ denote the set of inertial equivalence classes of  tamely ramified representations of $\WD_F$, and
let $\fR(\WD_F^{\tame})_{I_F,n}$ be the subset of inertial equivalence classes of  $n$-dimensional tamely ramified representations of $\WD_F$.
From \eqref{eqna:bij} we obtain a bijection
\begin{eqnarray} \label{eqn:bij-total}
\fP(\fX)^\Gamma&\longrightarrow&\fR(\WD_F^{\tame})_{I_F}\cr
\underline{\lambda}&\mapsto&\bigoplus_{[\theta]\in\Gamma\backslash \fX}\spe_F(\lambda_\theta)\otimes\rho(\theta)_{I_F},
\end{eqnarray}
which restricts to a bijection
\begin{equation} \label{eqn:bij-total-n}
\fP_n(\fX)^\Gamma\longrightarrow\fR(\WD_F^{\tame})_{I_F,n}.
\end{equation}

\subsection{Representations of \texorpdfstring{$\GL_n(k)$}{\GL_n(k)}} \label{sec:GLnk}
Let  $F$ be a non-archimedean local field with residue field $k$. We denote by $\fR_\CC(\GL_n(k))$  the category of representations of $\GL_n(k))$ with complex coefficients, 
and by $\fR_\CC(\GL_n(F))$ the category of \textit{smooth} (i.e., such that the stabilizer of any vector is open) representations of $\GL_n(F))$ with complex coefficients. 
Let $\FF$ be either $k$  or $F$. We denote by $\Irr_\CC(\GL_n(\FF))$ the irreducible objects in $\fR_\CC(\GL_n(\FF))$.

\subsubsection{The Harish-Chandra theory}
Given a sequence of positive integers $n_1$, $\ldots$, $n_r$ such that $n_1+\cdots+n_r=n$, we define $P(n_1,\ldots,n_r)$  to be the parabolic subgroup of $\GL_n(\FF)$ with Levi decomposition 
\begin{equation}
P(n_1,\ldots,n_r)=L(n_1,\ldots,n_r)\rtimes  U(n_1,\ldots,n_r),
\end{equation}
where 
\begin{equation}
L(n_1,\ldots,n_r):=\left\{\diag(g_1,\ldots,g_r)\,:\,g_i\in\GL_{n_i}(k), \text{for $1\le i\le r$}\right\},
\end{equation}
\begin{equation}
U(n_1,\ldots,n_r):=\left\{\left(\begin{matrix} 
\rI_{n_1}&\ast&\ast&\ast\cr
&\rI_{n_2}&\ast&\ast\cr
&&\ddots&\ast\cr
&&&\rI_{n_r}
\end{matrix} \right)\right\}.
\end{equation}
We set $U_n:=U(1,\ldots,1)$.

Given representations $\pi_1$, $\ldots$, $\pi_r$ of $\GL_{n_1}(\FF)$, $\ldots$, $\GL_{n_r}(\FF)$, respectively, we write $\sigma:=\pi_1\otimes\cdots\otimes\pi_r$, and denote by $\widetilde{\sigma}$ its inflation to $P$. We define
the parabolically induced representation
\begin{equation} \label{eqn:rI}
\rI(\pi_1,\ldots,\pi_r):=\Ind_{P(n_1,\ldots,n_r)}^{\GL_n(\FF)}(\widetilde{\sigma}).
\end{equation}
We write
\begin{equation} \label{eqn:irred}
\pi_1\odot\cdots\odot\pi_r:=\rI(\pi_1\otimes\cdots\otimes\pi_h).
\end{equation}
The operation $\odot$ is commutative and associative. 

A representation in $\Irr_\CC(\GL_n(\FF))$ is  said to be \textit{cuspidal} if it is not isomorphic to a quotient of any parabolically induced representation from a proper Levi subgroup of $\GL_n(k))$.
We denote by $\fB(\GL_n(k))$ the set of  $\GL_n(k)$-conjugacy classes of pairs $(L,\sigma)$, where $L$ is a Levi subgroup of $\GL_n(k)$ and $\sigma$ an irreducible cuspidal representation of $L$.
The category $\fR_\CC(\GL_n(k))$ decomposes as a direct product of full subcategories as follows:
\begin{equation}
\fR_\CC(\GL_n(k))=\prod_{\ft\in\fB(\GL_n(k))}\fR^\ft_\CC(\GL_n(k)),
\end{equation}
where a representation $\pi$ is an object of $\fR^\ft_\CC(\GL_n(k))$ if and only if every irreducible subquotient of $\pi$ occurs in $R_{\bL'}^{\GL_n}(\sigma')$, with $\bL'$ is a  $k$-rational Levi subgroup of a $k$-rational parabolic subgroup of 
$\GL_n(\kbar)$ and $\sigma'$ an irreducible cuspidal representation of $L':=\bL'(k)$ such that $(L',\sigma')\in\ft$.

In particular, we obtain the following partition of the set $\Irr_\CC(\GL_n(k))$
\begin{equation}
\Irr_\CC(\GL_n(k))=\bigsqcup_{\ft\in\fB(\GL_n(k))}\Irr^\ft_{\CC}(\GL_n(k)), 
\end{equation}
where $\Irr^\ft_{\CC}(\GL_n(k))$ is the set of irreducible objects of $\fR^\ft_\CC(\GL_n(k))$.

Let $\ft=(L,\sigma)\in\fB(\GL_n(k))$. Let $\bL$ be the $k$-rational Levi subgroup of $\GL_n(\kbar)$ with $L$ as its group of $k$-rational points. We define 
\begin{equation}
\cH(\GL_n(k),\sigma):=\End_{\CC\GL_n(k)}(R_\bL^{\GL_n}(\sigma)),
\end{equation}
where $R_\bL^{\GL_n}(\sigma)$ denote the parabolically induced representation of $\sigma$.  The parabolic (or Harish-Chandra) induction is a special case of the Lusztig induction $R_\bM^\bG$, where $\bM$ is  $k$-rational Levi subgroup of an arbitrary (that is, non necessarily $k$-rational) parabolic subgroup of a reductive 
connected algebraic group $\bG$ defined over $k$. 

There is an equivalence of categories
\begin{equation}
\fR_\CC^\ft(\GL_n(k))\,\longrightarrow\,\Mod-\cH(\GL_n(k),\sigma).
\end{equation}
The algebra $\cH(\GL_n(k),\sigma)$ is an Hecke-Iwahori algebra with Weyl group 
\begin{equation}
W(L,\sigma):=\rN_{\GL_n(k)}(L,\sigma)/L
\end{equation}
(see for instance \cite[Chap.~8]{Lus}).
By \eqref{eqn:equiv-cat}, we have a bijection between the simple modules of $\cH(\GL_n(k),\sigma)$ and the set $\Irr_\CC^\ft(\GL_n(k))$ of all the irreducible objects of $\fR_\CC^\ft(\GL_n(k))$,  and  hence  a bijection between isomorphism classes of irreducble  representations of $W(L,\sigma)$ and $\Irr_\CC^\ft(\GL_n(k))$
\begin{eqnarray} \label{eqn:piE}
\Irr(W(L,\sigma))&\to &\Irr_\CC^\ft(\GL_n(k))\cr
E&\mapsto&\pi_E.
\end{eqnarray}

The Weyl group $\rN_{\GL_n(k)}(L,\sigma)/L$ can be described as follows.
We first consider the case where $L\simeq\GL_m(k)^{\times e}$ with $me=n$ (that is, $L$ is $\GL_n(k)$-conjugate to $L(m,\ldots,m)$) and $\sigma=\sigma_0\otimes\cdots\otimes\sigma_0=:\sigma_0^{\otimes e}$,
where $\sigma_0$ is an irreducible cuspidal representation of $\GL_m(k)$. 
Then $W(L,\sigma)$ is the symmetric  group $S_e$.

In the general case,  we have
\begin{equation} \label{eqn:Levi}
L\simeq\GL_{m_1}(k)^{\times e_1}\times\cdots\times\GL_{m_h}(k)^{\times e_h}\quad\text{and}\quad
\sigma=\sigma_1^{\otimes e_1}\otimes\cdots\otimes\sigma_h^{\otimes e_h},
\end{equation}
where $\sigma_i\in\Irr_\CC(\GL_{m_i}(k))$ is cuspidal for every $i\in\{1,\ldots,h\}$ and the representations $\sigma_i$ and $\sigma_j$ are not isomorphic for any $i\ne j$.
We have 
\begin{equation} \label{eqn:W}
W(L,\sigma)=S_{e_1}\times\cdots\times S_{e_h}.
\end{equation}
The isomorphism classes of irreducible representations of $W(L,\sigma)$ are hence pa\-ra\-me\-tri\-zed by $h$-tuples $(\lambda_{e_1},\ldots,\lambda_{e_h})$, 
where $\lambda_{e_i}$ is a partition of $e_i$, for each $i\in\{1,\ldots,h\}$. We will denote by $E(\lambda_{e_1},\ldots,\lambda_{e_h})\in\Irr(W(L,\sigma))$
the representation parametrized by $(\lambda_{e_1},\ldots,\lambda_{e_h})$.

\subsubsection{Parametrization of the representations} 
Let $\Irr_\cusp(\GL_n(k))$ denote the set of isomorphism classes of irreducible cuspidal representations of $\GL_n(k)$.
Green's results \cite[p.431]{Gr} show the existence of a bijection between $\Gamma$-orbits $[\theta]$ of regular characters $\theta$ of 
$k_n^\times$ (i.e., $\Gamma$-orbits of degree $n$) and isomorphism classes of irreducible cuspidal representations $\pi_{[\theta]}$ of $\GL_n(k)$, unique up to isomorphism, such that
\begin{equation}
\tr(\pi_{[\theta]})=(-1)^{n-1}\sum_{\gamma\in\Gamma_n}\theta({}^{\gamma }x),
\end{equation}
for any $x\in k_n^\times$ of degree $n$ over $k$, where $k_n^\times$ is considered as a maximal torus in $\GL_n(\kbar)$.
The representation $\pi_{[\theta]}$ is up to a sign a Deligne-Lusztig character $R_\bT^{\GL_n(\theta)}$, where $\bT$ is an elliptic maximal torus of $\GL_n(\kbar)$ that is isomorphic to $k_n^\times$.

Let $L$ be a Levi subgroup of $\GL_n(k)$ and $\sigma$ an irreducible cuspidal representation of $L$. We keep the notation of \eqref{eqn:Levi}.
By applying the above bijection to the irreducible cuspidal representations $\sigma_1$, $\ldots$, $\sigma_h$,   we have for every $i\in\{1,\ldots,h\}$,
\begin{equation} \label{eqn:thetai}
\sigma_i=\pi_{[\theta_i]} \quad \text{where $\theta_i$ is a regular character of $k_{m_i}^\times$}.
\end{equation}
By combining \eqref{eqn:thetai} with the bijection \eqref{eqn:piE}, we attach to  each  $(\theta,\lambda)$, where $\theta:=(\theta_1,\ldots,\theta_h)\in\fX_{m_1}^\reg\times\cdots\times\fX_{m_h}^\reg$ and $\lambda:=(\lambda_{e_1},\ldots,\lambda_{e_h})$, with $\lambda_i$ a partition of $e_i$, and 
$e_1m_1+\cdots+e_hm_h=n$, the representation
\begin{equation} \label{eqn:piHCL}
\pi_{\theta,\lambda}:=\pi_{E(\lambda_{e_1},\ldots,\lambda_{e_h})}\in\Irr^\ft(\GL_n(k)),
\end{equation}
where $\ft=(L,\sigma)$ with 
\begin{equation}
L\simeq\GL_{m_1}(k)^{\times e_1}\times\cdots\times\GL_{m_h}(k)^{\times e_h}\quad\text{and}\quad 
\sigma=\pi_{[\theta_1]}^{e_1} \otimes \cdots\otimes\pi_{[\theta_h]}^{e_h}.
\end{equation}

\subsubsection{The Lusztig classification}
We keep the notation of the previous sections. In particuler, $\bL$ is a $k$-rational Levi subgroup of a $k$-rational parabolic subgroup of $\GL_n(\kbar)$.
We recall that, by \cite[Corollary~7.7]{DL}, for every irreducible representation $\pi$ of $\GL_n(k)$, there exists a $k$-rational maximal torus $\bT$ of $\GL_n(\kbar)$ and a character $\theta$ of $T:=\bT(k)$ such that the character of $\pi$ (still denoted by $\pi$) occurs in the Deligne-Lusztig (virtual) character $R_\bT^{\GL_n}(\theta)$, i.e.~such that
\begin{equation} 
\langle\pi, R_\bT^{\GL_n}(\theta)\rangle_{\GL_n(k)}\ne 0,
\end{equation} where $\langle\;,\;\rangle_{\GL_n(k)}$ is the usual scalar 
product on the space of class functions on $\GL_n(k)$: 
\begin{equation} \label{eqn:scalar product}
\langle f_1,f_2\rangle_{\GL_n(k)} = |\GL_n(k)|^{-1}\sum_{g\in\GL_n(k)}f_1(g)\,\overline{f_2(g)} .
\end{equation}
If $\theta=1$ (i.e.~the trivial character of $T$), then the representation $\pi$ is called \textit{unipotent}. 

The $\GL_n(k)$-conjugacy classes of pairs $(\bT,\theta)$ as above are in one-to-one correspondence with the $\GL_n(k)$-conjugacy classes of pairs $(\bT^\vee,s)$, where $s$ is a semisimple element of $L$, 
and $\bT$ is a $k$-rational maximal torus of $\GL_n(\kbar)$ containing $s$ (for a general reductive algebraic group $\bG$ defined over $k$, we need to consider a semisimple element in a maximal torus $\bT^\ast$ of the group $\bG^\ast$, where $\bG^\ast$ is the group defined over $k$, with root datum dual to that of $\bG$; for $\bG=\GL_n$, we have $\bG^\ast=\bG$ and $\bT^\ast=\bT$).
Then the \textit{Lusztig series} $\mathcal{E}(\GL_n,s)$ is defined as the set of the representations $\pi\in\Irr(\GL_n(k))$ such that $\pi$ occurs in $R_\bT^{\GL_n}(\theta),$ where $(\bT,\theta)_{\GL_n(k)}$ corresponds to $(\bT,s)_{\GL_n(k)}$.
By definition, $\mathcal{E}(\GL_n,1)$ consists only of unipotent representations. 

By \cite[\S10]{DL} and \cite[(8.4.4)]{Lus}, the set $\Irr_\CC(\GL_n(k))$ decomposes into a disjoint union:
\begin{equation} \label{eqn:Lusztig-series}
\Irr_\CC(\GL_n(k))=\bigsqcup_{(s)}\mathcal{E}(\GL_n,s),
\end{equation}
where $(s)$ is the $\GL_n(k)$-conjugacy class of a semisimple element $s$ of $\GL_n(k)$.

By  \cite[Theorem~4.23]{Lus}, there is a bijection
\begin{equation}\label{Lusztig-unipotent-decomposition}
 J_{s}^{\GL_{n}}\colon  \mathcal{E}(\GL_n(k),s)\xrightarrow{1-1} \mathcal{E}(\Cent_{\GL_n(\kbar)}(s)(k),1), \quad
    \pi\mapsto\pi^\unip.
\end{equation}
Moreover, the bijection $J_{s}^{\GL_n}$ can be chosen so that the following diagram is commutative (see for instance \cite[Corollary~4.7.6]{GeckMalle}):
\begin{equation} \label{eqn:commutes}
\begin{tikzcd}
\ZZ\cE(L,s) \arrow[]{d}[swap]{R_\bL^{\GL_n}}\arrow[]{r}{J_{s}^{L}}&\ZZ\cE(\Cent_{\bL}(s)(k),1)\arrow[]{d}{R_{\Cent_{\bL}(s)}^{\Cent_{\GL_n}(s)}}\\
\ZZ\cE(\GL_n(k),s) \arrow[]{r}[swap]{J_{s}^{\GL_{n}}} & \ZZ\cE(\Cent_{\GL_n(\kbar)(k)}(s),1)
\end{tikzcd}.
\end{equation}
Let $s$ be a semisimple element in $\GL_n(k)$ such that the representation 
$\pi:=\pi_{\theta,\lambda}$ defined in \eqref{eqn:piHCL} belongs to the Lusztig series $\cE(\GL_n(k),s)$. We have $s=s_1\cdots s_h$, where, for each $i\in\{1\ldots,h\}$,  the element $s_i$
a semisimple element in $\GL_{e_im_i}(\kbar)$ such that
 \begin{equation} \label{eqn:centra}
 \Cent_{\GL_{e_im_i}}(s_i)=\GL_{e_i}(k_{m_i}).
 \end{equation}. 
Hence we obtain
\begin{equation}
\Cent_{\GL_n(\kbar)}(s)\simeq \GL_{e_1}(k_{m_1})\times\cdots\times\GL_{e_h}(k_{m_h}).
\end{equation}
 The unipotent representation $\pi^\unip$ of $\Cent_{\GL_n(\kbar)}(s)(k)$ decomposes as
 \begin{equation}
 \pi^\unip=\pi^\unip_1\otimes\cdots\otimes\pi^\unip_h,\quad\text{where $\pi_i^\unip\in \cE(\GL_{e_i}(k_{m_i}),1)$ for every $i\in\{1\ldots,h\}$.}
 \end{equation}
For every $i\in\{1,\ldots,h\}$, let $\pi_i$ be the irreducible representation of $\GL_{e_im_i}(k)$ in the Lusztig series $\cE(\GL_{e_im_i}(k),s_i)$ of the group $\GL_{e_im_i}(k)$ such that  
\begin{equation} \label{eqn:pi}
J_{s_i}^{\GL_{e_im_i}}(\pi_i)=\pi^\unip_i.
\end{equation}
Let $E(\lambda_{e_i})$ be the irreducible representation of $S_{e_i}$ which is parametrized by the partition $\lambda_i$ of $e_i$. We have
\begin{equation} \label{eqn:pii}
\pi_i=\pi_{E(\lambda_{e_i})}=\pi_{[\theta_i],\lambda_i}.
\end{equation}

\begin{remark} \label{rem:Mac}
{\rm  The representation $\pi_{[\theta_i],\lambda_i}$ is denoted by $(f_i^{\lambda_i})$, where $(f_i)=[\theta_i]$, in  \cite[\S1]{Mac}.}
\end{remark}

\begin{prop} \label{prop:irrred}
Let $M$ be a Levi subgroup of $\GL_n(k)$ such that
\begin{equation}
M\simeq\GL_{e_1m_1}(k)\times\cdots\times\GL_{e_hm_h}(k).
\end{equation}
Then the parabolically induced representation $R_M^{\GL_n}(\pi_1\otimes\cdots\otimes\pi_h)$ is irreducible.

\end{prop}
\begin{proof}
By combining \eqref{eqn:pi} and \eqref{eqn:commutes}, we obtain
\begin{equation} \label{eqn:compat}
R_M^{\GL_n}(\pi_1\otimes\cdots\otimes\pi_h)=(R_M^{\GL_n}\circ J_{s}^M)(\pi_1\otimes\cdots\otimes\pi_h)=(J_{s}^{\GL_n}\circ R_{\Cent_\bM(s)}^{\Cent_G(s)})( \pi^\unip).
\end{equation}
We have 
\[\Cent_\bM(s)=\Cent_{\GL_{e_1m_1}}(s_1)\times\cdots\times\Cent_{\GL_{e_hm_h}}(s_h),\]
and hence, it follows from \eqref{eqn:centra} that $\Cent_\bM(s)=\Cent_{\GL_n(\kbar)}(s)$. Thus, \eqref{eqn:compat} implies that $R_M^{\GL_n}(\pi_1\otimes\cdots\otimes\pi_h)$ is irreducible.
\end{proof}
\begin{cor} \label{cor:larep} We have
\begin{equation} 
\pi_{[\theta],\lambda_\theta}=\pi_{[\theta_1],\lambda_1}\odot\cdots\odot\pi_{[\theta_h],\lambda_h}.
\end{equation}
\end{cor}
\begin{proof}
Proposition~\ref{prop:irrred} implies that
\begin{equation} 
\pi_{[\theta],\lambda_\theta}=R_M^{\GL_n}(\pi_1\otimes\cdots\otimes\pi_h).
\end{equation}
\end{proof}

\begin{remark} {\rm The combination of  Remark~\ref{rem:Mac} and Corollary~\ref{cor:larep}  shows that the representation $\pi_{[\theta],\lambda_\theta}$ coincides with the representation denoted by
$(f_1)^{\lambda_1}\circ\cdots\circ (f_m)^{\lambda_m}$, with $m=h$, in \cite[\S1]{Mac}.}
\end{remark}

By \eqref{eqn:piHCL}, we have obtained a bijection  
\begin{eqnarray} \label{eqn:Macdonald-bijection}
\fP_n(\fX)^\Gamma&\longrightarrow&\Irr(\GL_n(k))\cr
\underline{\lambda}&\mapsto &\bigodot_{[\theta]\in \Gamma\backslash \fX}\pi_{[\theta],\lambda_\theta}.
\end{eqnarray}
We write $\GL_0(k):=\{1\}$, and set
\begin{equation}
\Irr(\GL(k)):=\bigsqcup_{n\in \ZZ_{\ge 0}}\Irr(\GL_n(k)).
\end{equation}
Hence, we have a bijection  
\begin{eqnarray}
\fP(\fX)^\Gamma&\longrightarrow&\Irr(\GL(k))\cr
\underline{\lambda}&\mapsto &\bigodot_{[\theta]\in \Gamma\backslash \fX}\,\pi_{[\theta],\lambda_\theta}.
\end{eqnarray}

\subsection{Definition of the Macdonald correspondence}
The bijection obtained by composing the inverse of the bijection \eqref{eqn:Macdonald-bijection} with the bijection \eqref{eqn:bij-total-n} is the bijection defined by Macdonald in \cite[\S1]{Mac}:
\begin{equation} \label{eqn:MC}
\cM_n\colon\Irr(\GL_n(k))\overset{1-1}{\longrightarrow}\fR(\WD_F^{\tame})_{I_F,n}.
\end{equation}
We call $\cM_n$ the \textit{Macdonald correspondence}. 

We have
\begin{equation} \label{eqn:observation}
\omega_\pi=\det(\cM_n(\pi)),\quad \text{for every $\pi\in \Irr(\GL_n(k)$,}
\end{equation}
where $\omega_{\pi}$ denotes the central character of $\pi$, and the determinant character of $\cM_n(\pi)$ is considered as a function of  
$k^\times=\fo_F^\times/1+\fp_F\hookrightarrow (W^{\tame})^{\abel}$ (see \cite[(1.2)]{Mac}).

By construction, the map $\cM_n$ restricts to a bijection 
\begin{equation} \label{eqn:MC-cusp}
\cM_n^\cusp\colon\Irr_\cusp(\GL_n(k))\overset{1-1}{\longrightarrow}\Irr(\WD_F^{\tame})_{I_F,n}
\end{equation}
between the isomorphism classes of irreducible cuspidal representations of $\GL_n(k)$ and the inertial equivalence classes of $n$-dimensional irreducible tamely ramified representations of $\WD_F$.

\subsection{Definition and preservation of the $\epsilon$-factors}
Let $\psi\colon k\to \CC^\times$ be a nontrivial additive character.
Let $U_n$ denote the upper triangular unipotent subgroup of $\GL_n(k)$. We define a character $\psi_n\colon U_n\to\CC^\times$ by 
\begin{equation} 
\psi_n\left(\begin{smallmatrix}
1&a_1&\ast&\ldots&\ast\cr
&1&a_2&\ldots&\ast\cr
&&\ddots&\ddots&\vdots\cr
&&&1&a_{n-1}\cr
&&&&1
\end{smallmatrix}\right):=\psi(a_1+a_1+\cdots+a_{n-1}).
\end{equation}
A finite dimensional representation $\pi$ of $\GL_n(k)$ is said to be \textit{generic} if 
\begin{equation} \Hom_{U_n}(\pi|_{U_n},\psi_n)\ne 0.
\end{equation} 
The representation $\pi$ is generic if and only if there exists a nonzero vector $v\in V_\pi$ such that $\pi(u)v=\psi_n(v)$ for every $u\in U_n$. Such a vector $v$ is called a \textit{Whittaker vector} (with respect to $\psi_n$). 
We say that $\pi$  is \textit{of Whittaker type} if $\pi$ is generic and the subspace spanned by its Whittaker vectors is one-dimensional.

Let $V$ be $k$-vector space of dimension $n$. We set $A:=\End_k(V)$ and let $\cC(A)$ be the space of complex valued functions on $A$. 
Following \cite{Springer} or \cite{Mac2}, we introduce the notion of the Fourier transform and zeta function of complex representations of $G:=\GL(V)$ as follows. For $\phi\in\cC(A)$, the Fourier transform $\widehat\phi$ is defined by
\begin{equation}
\widehat\phi(a):=q^{-n^2/2}\sum_{a'\in A} \phi(a')\,\Tr(aa').
\end{equation}
We have $\widehat{\widehat\phi}(a)=\phi(a)$ for every $a\in A$. 

For a finite dimensional complex representation $\pi$ of $G$, and for $\phi\in\cC(A)$, we define the zeta function $Z(\pi,\psi)$ by
\begin{equation}
Z(\pi,\psi):=\sum_{g\in G} \phi(g)\pi(g).
\end{equation}
For every $a\in A$, we set
\begin{equation}
W(\pi,\psi;a):=q^{-n^2/2}\sum_{g\in G}\Tr(\pi(ga))\pi(a).
\end{equation}
We have 
\begin{equation}
Z(\pi,\psi)=\sum_{a\in A} \widehat\phi(-a)\cdot W(\pi,\psi;a).
\end{equation}
For $g'\in G$, we have
\[W(\pi,\psi;g'a)=W(\pi,\psi;a)\pi(g')^{-1}\quad\text{and}\quad W(\pi,\psi;g'a)=\pi(g')^{-1}W(\pi,\psi;a).\]
It implies that  $\pi(g')$ commutes with $W(\pi,\psi;1)$. Hence, if $\pi$ is irreducible, there exists $w(\pi,\psi)\in \CC$ such that
\begin{equation}
W(\pi,\psi;1)=w(\pi,\psi)\,\pi(1).
\end{equation}
The $\epsilon$-factor $\epsilon(\pi,\psi)$ is defined to be 
\begin{equation} \label{eqn:epsilon}
\epsilon(\pi,\psi):=w(\pi^\vee,\psi),
\end{equation}
where $\pi^\vee$ is the contragredient representation of  $\pi$.

Let $\Psi$ be nontrivial additive character of $F$. We suppose that $\fP_F\subset \ker\Psi$ and $\fo_F\not\subset\ker \Psi$. By restriction to $\fo_F$ and reduction to $\fp_F$, the character $\Psi$ induces an nontrivial 
additive character $\psi$ of $k$. We fix the Haar measure on $F$ normalized such that $\fp_F$ has volume $q^{-1/2}$.

Deligne has attached in \cite[(5.1) and \S8.12]{Del}  to every Weil-Deligne representation $\varphi:=(\rho,N)$ of $\WD_F$ a nonzero constant $\epsilon_0(\varphi,\Psi)$. When $\varphi$ is tamely ramified, this constant 
is computed in \cite[\S5.16]{Del}, and Macdonald has proved in \cite[\S3]{Mac} that it depends only on the restriction of $\rho$ to $I_F$. Moreover, Macdonald has established in \cite[\S3]{Mac} that the following equality holds
\begin{equation} \label{eqn=ee0}
\epsilon(\pi,\psi)=\epsilon_0(\cM_n(\pi),\Psi), \quad\text{for every $\pi\in\Irr(\GL_n(k))$.}
\end{equation}

\section{Gamma factors for pairs}
Let $n_1$ and $n_2$ be positive integers, and let $(\pi_i,V_{\pi_i})\in\Irr(\GL{n_i}(k))$ for $i\in\{1,2\}$. Let  $\sw_{\pi_1,\pi_2}\colon V_{\pi_1}\otimes V_{\pi_2}\to V_{\pi_2}\otimes V_{\pi_1}$ denote a linear map defined on pure tensors by 
\begin{equation} 
\sw_{\pi_1,\pi_2}(v_1\otimes v_2):=v_2\otimes v_1, \quad \text{for all $v_1\in V_{\pi_1}$ and $v_2\in V_{\pi_2}$.}
\end{equation}
For a function $f\colon \GL_{n_1+n_2}(k)\to V_{\pi_1\times\pi_2}$, we denote by $\overline f\colon \GL_{n_1+n_2}(k)\to V_{\pi_2\times\pi_1}$ the function defined by
\[\overline f(g):=\sw_{\pi_1,\pi_2}(f(g)),\quad \text{for all $g\in \GL_{n_1+n_2}(k)$.}\]
We write $\widehat w_{n_1,n_2}:=\left(\begin{smallmatrix}0&\rI_{n_2}\cr \rI_{n_1}& 0\end{smallmatrix}\right)$.

We consider the intertwining operator $\cA_{\pi_1,\pi_2}\colon\rI(\pi_1,\pi_2)\to\rI(\pi_2,\pi_1)$ defined by
\begin{equation} \label{eqn:intertwU}
\cA_{\pi_1,\pi_2}(f)\colon g\mapsto\sum_{u\in U(n_1,n_2)}\overline f(\widehat w_{n_1,n_2}ug).
\end{equation}

Let $\pi_1$ and $\pi_2$ be representations of Whitaker type of $\GL_{n_1}(k)$ and $\GL_{n_2}(k)$, respectively. 
By \cite[Theorem~5.5]{SiZi-Steinberg} the parabolically induced representation $\rI(\pi_1,\pi_2)$ is also of Whitaker type. 

Let $v_{\pi_i,\psi}$ be a nonzero Whittaker vector for $\pi_i$, for $i=1,2$.
We set $n_{1,2}:=n_1+n_2$ and 
\[v_{\pi_1,\pi_2,\psi}:=\begin{cases}
\psi_{n_{1,2}}(u)(\widetilde{\pi_1\otimes\pi_2})(p)\,v_{\pi_1,\psi}\otimes v_{\pi_2,\psi}&\text{if $g=p \widehat w_{n_1,n_2}u$, $p\in P_{n_1,n_2}$, $u\in U_{n_{1,2}}$}\cr
0&\text{otherwise.}
\end{cases}\]
Then, $v_{\pi_1,\pi_2,\psi}$ is a non-zero Whittaker vector of $\rI(\pi_1,\pi_2)$.

\smallskip

Following \cite[Definition~3.1]{SoZe}, we define the  \textit{Shahidi gamma factor} of $\pi_1$ and $\pi_2$ with respect to $\psi$ to be the unique complex number $\overline\gamma^\Sh(\pi_1\times\pi_2,\psi)$ satisfying
\begin{equation} \label{eqn:gamma}
\cA_{\pi_1,\pi_2}(v_{\pi_1,\pi_2,\psi})=\overline\gamma^\Sh(\pi_1\times\pi_2,\psi)\cdot v_{\pi_2,\pi_1,\psi}.
\end{equation}

By \cite[\S4.A]{SoZe}, the \textit{normalized Shahidi gamma factor} is 
\begin{equation} \label{eqn:Shgamma}
\gamma^\Sh(\pi_1\times\pi_2,\psi):=q^{-n_1n_2/2}\cdot\overline\gamma^\Sh(\pi_1\times\pi_2,\psi).
\end{equation}
The relation of normalized Shahidi gamma factors and Jacquet-Piatetski-Shapiro-Shalika factors is the following (see \cite[Corollary~3.15]{SoZe}):
\begin{equation} \label{eqn:S-JPSS}
\gamma^\Sh(\pi_1\times\pi_2,\psi)=q^{(n_1-n_2-1)n_2}\cdot\omega_{\pi_2}(-1)\cdot\gamma(\pi_1\times\pi_2^\vee,\psi)
\end{equation}
in either of the following cases
\begin{enumerate}
\item $n_1>n_2$,
\item
$n_1=n_2$ and $\pi_2$ is cuspidal.
\end{enumerate}

\begin{defn} \label{defn:eps-for-pairs} {\rm
If $\sigma_1\in\Irr(\GL_{n_1}(k))$ and $\sigma_2\in\Irr(\GL_{n_2}(k))$ are cuspidal,  we define the $\epsilon$-factor for pairs $\epsilon(\sigma_1\times\sigma_2,\psi)$ by 
\begin{equation} 
\epsilon(\sigma_1\times\sigma_2,\psi):=\omega_{\sigma_2}(-1)^{n_2}\cdot\gamma^\Sh(\sigma_1\times\sigma_2^\vee,\psi).
\end{equation}}
\end{defn}

\section{More properties of the Macdonald correspondence} \label{sec:properties}
\subsection{Compatibility with the depth-zero local Langlands correspondence}
Let $\Irr(\GL_n(F))_0$ denote the set of isomorphism classes of depth-zero irreducible representations of $\GL_n(F)$, and $\fR(\WD_F^{\tame})_{n}$ the category of $n$-dimensional tamely ramified representations of  the Weil-Deligne group $\WD_F$.
Silberger and Zink have established in \cite[Appendix]{SiZi2} the commutativity of the following diagram:
\begin{equation}  \label{Diagram-SZ}
\begin{tikzcd}
{\Irr(\GL_n(F))_0} \arrow[]{d}[swap]{p_0}\arrow[]{r}{\LLC_n}&\fR(\WD_F^{\tame})_{n}\arrow[]{d}{{}^Lp_0}\\
\Irr(\GL_n(k))\arrow[]{r}[swap]{\cM_n} & \fR(\WD_F^{\tame})_{I_F,n}
\end{tikzcd},
\end{equation}
where $p_0$ is the map sending an irreducible depth-zero supercuspidal representation of $\GL_n(F)$ to the irreducible cuspidal representation from which it was constructed, and ${}^Lp_0$ is the canonical projection.
The Macdonald correspondence is also compatible with the Shintani descent (see \cite[Theorem~5.1]{SiZi2}).

\subsection{Preservation of the \texorpdfstring{$\epsilon$}{epsilon}-factors for pairs}
In \cite{YZ2},  Ye and Zelingher defined epsilon factors for  pairs for irreducible representations of finite general linear groups using the Macdonald correspondence. Following \cite[\S3]{YZ2},  for $\pi_1\in\Irr(\GL_{n_1}(k))$ 
and $\pi_2\in\Irr(\GL_{n_2}(k))$  we set
\begin{equation}
\epsilon_0(\pi_1\times\pi_2,\psi):=\epsilon_0(\cM_{n_1}(\pi_1)\otimes\cM_{n_2}(\pi_2),\psi).
\end{equation}

\begin{thm} \label{thm:gammas} Suppose $n_1\ge n_2$.
Let $\sigma_1$ and $\sigma_2$ be irreducible cuspidal representations of $\GL_{n_1}(k)$ and $\GL_{n_2}(k)$, respectively. 
Then the following equality holds
\[\epsilon(\sigma_1\times\sigma_2,\psi)=\epsilon_0(\cM_{n_1}(\sigma_1)\otimes\cM_{n_2}(\sigma_2),\psi).\]
\end{thm}
\begin{proof}
By combining \eqref{eqn:S-JPSS} with \cite[Theorem~4.4]{YZ2} (for the case $n_1>n_2$) and \cite[Theorem~2.18]{Zeli} (for the case $n_1=n_2$), we obtain
\[\gamma^\Sh(\sigma_1\times\sigma_2^\vee,\psi)=\omega_{\sigma_2}(-1)^{n_2}\cdot\epsilon_0(\sigma_1\times\sigma_2,\psi).\]
The result follows.
\end{proof}

\section{Characterization of the Macdonald correspondence} \label{subsec:Macdo-charac}
We define 
\begin{eqnarray}
\Irr(\WD_F^{\tame})_{I_F,n}&\longrightarrow&\Irr_\cusp(\GL_n(k))\cr
\rho&\mapsto& \pi_\rho 
\end{eqnarray} 
to be the inverse map of $\cM_n^\cusp$.

We recall the following result, that was was obtained by Nien in 2014:
\begin{thmconverse} {\rm \cite[Theorem~3.9]{Nie-2014} }
Let $n$ be an integer $\ge 2$, and let $\pi$ and $\pi'$ be two  irreducible cuspidal representations of $\GL_{n}(k)$ with the same central character, such that,
for every integer $m$ such that $1\le m\le n/2$, and every  irreducible generic representation $\tau$ of $\GL_{m}(k)$, we have the equality
\begin{equation} 
\gamma(\pi\times\tau,\psi)=\gamma(\pi'\times\tau,\psi).
\end{equation}
Then the representations $\pi$ and $\pi'$ are equivalent.
\end{thmconverse}

\begin{cor} \label{cor:cuspLC}
Let $n$ be an integer such $n\ge 2$, and let $\pi$ be an irreducible cuspidal representation of $\GL_n(k)$ and $\rho$ be an irreducible $n$-dimensional representation of $W_F$ such that $\pi_\rho$ and $\pi$ have same central character and
\[\epsilon(\pi\times \pi_{\rho'},\psi)=\epsilon(\rho\otimes \rho',\psi),\]
for any integer $m$ such that $1\le m\le n/2$ and every  irreducible $n$-dimensional representation $\rho'$ of $W_F$. 

Then the representation $\pi$ is equivalent to $\pi_\rho$.
\end{cor}
\begin{proof} It follows from the above Converse Theorem by using  Definition~\ref{defn:eps-for-pairs} and \eqref{eqn:S-JPSS}.
\end{proof}

\begin{thm} \label{thm:cuspLC} Let $(\cT_n)_{n\ge 1}$ be a collection of bijections 
\begin{eqnarray*} \cT_n\colon&\Irr(\WD_F^{\tame})_{I_F,n}\longrightarrow\Irr_\cusp(\GL_n(k))\\
&\rho\mapsto \tau_\rho
\end{eqnarray*}
satisfying the following properties:
\begin{enumerate}
\item[{\rm(1)}] for every integer $n\ge 1$, and every $\rho\in\Irr(\WD_F^{\tame})_{I_F,n}$, 
\[\omega_{\tau_\rho}=\det(\rho),\]
\item[{\rm(2)}]  for all integers $n_1,n_2\ge 1$, and every $\rho_1\in\Irr(\WD_F^{\tame})_{I_F,n_1}$ and every  $\rho_2\in\Irr(\WD_F^{\tame})_{I_F,n_2}$, 
\[\epsilon(\tau_{\rho_1}\times \tau_{\rho_2},\psi)=\epsilon(\rho_1\otimes\rho_2,\psi).\]
\end{enumerate}
Then $\cT_n=\cM_n^\cusp$.
\end{thm}
\begin{proof}
By \eqref{eqn:observation}, the collection of bijections $((\cM_n^\cusp)^{-1})_{n\ge 1}$ satisfies  the property (1). As above, we write $\pi_\rho:=\cM_n^{-1}(\rho)$. 
Thus, we have, for every integer $n\ge 1$, and every every $\rho\in\Irr(\WD_F^{\tame})_{I_F,n}$, 
\[\omega_{\pi_\rho}=\det(\rho).\] 
It follows that the representations $\pi_\rho$ and $\tau_\rho$ have the same central characters:
\begin{equation} \label{eqn:centraux}
\omega_{\pi_\rho}=\omega_{\tau_\rho}.
\end{equation}
By Theorem~\ref{thm:gammas}, the collection of bijections $(\cM_n^{-1})_{n\ge 1}$ satisfy the property (2). Thus, we have, for all integers $n_1,n_2\ge 1$, 
and every $\rho_1\in\Irr(\WD_F^{\tame})_{I_F,n_1}$ and every  $\rho_2\in\Irr(\WD_F^{\tame})_{I_F,n_2}$,
\[\epsilon(\pi_{\rho_1}\times\pi_{\rho_2},\psi)=\epsilon(\rho_1\otimes\rho_2,\psi).\]
Hence, we get
\begin{equation} \label{eqn:epsilons-equal}
\epsilon(\pi_{\rho_1}\times\pi_{\rho_2},\psi)=\epsilon(\tau_{\rho_1}\times\tau_{\rho_2},\psi).
\end{equation}
Equations \eqref{eqn:centraux} and \eqref{eqn:epsilons-equal} allow us to apply Corollary~\ref{cor:cuspLC}. We obtain that 
\[\tau_\rho=\pi_\rho,\quad\text{for every integer $n\ge 1$, and every $\rho\in\Irr(\WD_F^{\tame})_{I_F,n}$.}\]
\end{proof}

\end{document}